\documentclass[12pt]{article}
\usepackage{amsmath}
\usepackage{graphicx,psfrag,epsf}
\usepackage{enumerate}
\usepackage{natbib}
\usepackage{amsmath,amsthm,amsfonts,epsfig,amssymb,natbib,eucal,eufrak}

\usepackage{booktabs}
\usepackage{multirow}
\usepackage{siunitx}
\usepackage{qtree}
\usepackage{hyperref}

\usepackage{float}
\restylefloat{table}
\usepackage{color}

\usepackage{pdfpages}
\usepackage{wrapfig}

\usepackage{comment}

\usepackage{tikz}

\newtheorem{thm}{Theorem}

\newtheorem{result}{Result}

\newtheorem{pro}{Proposition}

\newtheorem{cl}{Claim}

\newcommand{\blind}{0}

\def\baselinestretch{1.19}

\begin{document}

\def\spacingset#1{\renewcommand{\baselinestretch}%
{#1}\small\normalsize} \spacingset{1}

%%%%%%%%%%%%%%%%%%%%%%%%%%%%%%%%%%%%%%%%%%%%%%%%%%%%%%%%%%%%%%%%%%%%%%%%%%%%%%
%%%%%%%%%%%%%%%%%%%%%%%%%%%%%%%%%%%%%%%%%%%%%%%%%%%%%%%%%%%%%%%%%%%%%%%%%%%%%%

%%%%%%%%%%%%%%%%%%%%%%%%%%%%%%%%%%%%%%%%%%%%%%%%%%%%%%%%%%%%%%%%%%%%%%%%%%%%%%%%%%%5

\if0\blind
{\title{\bf Uniqueness of Maximum Scores in Countable-Outcome Round-Robin Tournaments}

%More on Round-Robin Tournament Models with a Unique Maximum Score}

\author
{
Gideon Amir
\thanks{Department of Mathematics, Bar-Ilan University, Ramat Gan 5290002, Israel.
Email: {\tt gideon.amir@biu.ac.il}.}
\and
Yaakov Malinovsky
\thanks{Department of Mathematics and Statistics, University of Maryland,
Baltimore County, Baltimore, MD 21250, USA.
Email: {\tt yaakovm@umbc.edu}.}
}
 \maketitle
} \fi
\if1\blind
{
  \bigskip
  \bigskip
  \bigskip
  \begin{center}
    {\LARGE\bf }
\end{center}
  \medskip
} \fi

\bigskip
\begin{abstract}
In this note, we extend a recent result on the uniqueness of the maximum score in a classical round-robin tournament to general round-robin tournament models with equally strong players, where the scores take values in $[0,\,1]$.
\end{abstract}
\noindent%

\noindent%
%\hfill {\tiny technometrics tex template (do not remove)}

\spacingset{1.45} % DON'T change the spacing!

\section{Introduction and Problem Statement}
In a round-robin tournament, each of $n$ players competes against each of the other $n-1$ players. When player
$i$ plays against player $j$, player $i$'s reward is a random variable $X_{ij}$. Let
\begin{equation}
\label{eq:not}
s_i(n)=\sum_{j=1, j\neq i}^{n}{X_{ij}}
\end{equation}
denote the score of player $i, 1\leq i \leq n$, after playing against all the other $n-1$ players.
We assume that all ${n \choose 2}$ pairs $\left(X_{ij}, X_{ji}\right), 1\leq i< j\leq n$ are independent.
We refer to $\left(s_1(n), s_2(n),\ldots, s_n(n)\right)$ as the score sequence of the tournament.

Tournaments provide a model for a statistical inference called paired comparisons.
Measuring players' strength in chess tournaments by modeling paired comparisons of strength has a long history,
first appearing in \cite{Z1929}.
An excellent introduction to the combinatorial and some probabilistic aspects (with citations to earlier works) is given in \cite{Moon1968}. The statistical aspects are presented in \cite{D1988} (also with citations to earlier works). Additional material on round-robin tournaments and related issues can be found in recent papers and the references therein \citep{R2021, MM2021}.

In a classic round-robin tournament (Model $M_1$) for $i\neq j$, $X_{ij}+X_{ji}=1$, $X_{ij}\in \left\{0, 1\right\}$, and $\mathbb{P}(X_{ij}=1)=\mathbb{P}(X_{ij}=0)=1/2$, i.e., each of $n$ players wins or loses a game against each of the
other $n-1$ with equal probabilities. 
The tournament can be represented by complete oriented graphs in which the vertices represent the players, and each pair of distinct vertices $i$ and $j$ is joined by an edge oriented from $i$ to $j$ or from $j$ to $i$ according to whether, $X_{ij}=1$ or $X_{ji}=1$, respectively.
Let $r_n(M_1)$ denote the probability that the tournament with $n$ labeled vertices has a unique
vertex with maximum score under model $M_1$.
\cite{E1967} stated, without proof, that in a classical round-robin tournament (Model $M_1$ ), the probability that there is a unique
vertex with the maximum score tends to $1$ as $n$ tends to infinity, i.e.,
\begin{equation}
\label{eq:C1}
{\displaystyle \lim_{n\rightarrow \infty}r_n(M_1)=1.}
\end{equation}
For values of $n$ that are very small, it is possible to compute $r_n(M_1)$ from the distribution of the nondecreasing score sequences. However, even for small $n$, this becomes computationally impossible due to the fact that the size of the sample space is $2^{n\choose 2}$. Indeed, \cite{E1967} employed the distribution of the nondecreasing score vector up to $n=8$, as provided by \cite{D1959}, to calculate the values of $r_1(M_1),\ldots, r_8(M_1)$. It is interesting to note that \cite{M1923} generated the score sequences and their frequencies up to $n=9$.
\cite{Z2016} extended MacMahon's work and generated the nondecreasing score sequences and their frequencies for $n\leq 15$ using the Maple program.
In a survey paper, \cite{G1984} notes that Epstein's problem on $r_n(M_1)$ remains unsolved.
In a recent publication, \cite{MM2024} proved \eqref{eq:C1}. They used a method which was introduced by \cite{EW1977} that considered the analogous problem for the vertices of maximum degree in a random labelled graph in which pairs of distinct vertices are joined by an edge with probability $1/2$. In the next paragraph, we outline this method, which was applied to model $M_1.$

Set $$s^{\star}_n=\max\left\{s_1(n),\ldots, s_n(n)\right\},$$ and
\begin{align}
\label{eq:W}
W_n(t)=\sum_{1\leq v<u \leq n }I(t<s_{u}(n)=s_{v}(n)).
\end{align}

Under model $M$, we need to identify the value $t_n$, if such a value exists, so that
\begin{align}
\label{eq:Step1}
\lim_{n\rightarrow \infty}\mathbb{P}\left(s^{\star}_n>t_n \right)=1
\end{align}
and
\begin{equation}
\label{eq:Step2}
{\displaystyle \lim_{n\rightarrow \infty}\mathbb{P}\left(W_n(t_n)=0\right)=1.}
\end{equation}

Combining, \eqref{eq:Step1} and \eqref{eq:Step2}, we obtain
\begin{equation*}
\label{eq:ModelM}
{\displaystyle \lim_{n\rightarrow \infty}r_n(M)=1.}
\end{equation*}

The first question that could be asked is whether a similar result holds for a more general model, e.g. a chess tournament model with equally strong players,
where a player scores one point for a game won and half a point for a game drawn (hereafter model $M_2$.)
\cite{M1923b} considered such a model motivated by the Master's Chess Tournament which was held in London in 1922.
In that work, MacMahon generated the nondecreasing score sequences and their frequencies for $n\leq 6$.
It is interesting to note that \cite{Z1929} proposed a maximum likelihood estimation of a player's strength in chess tournaments (see also \cite{DE2001}, pages 161-186.)
Zermelo applied his estimation technique in analyzing the 1924 New York Chess Tournament, a famous tournament featuring prominent players like José Raúl Capablanca, Emanuel Lasker, and Alexander Alekhine.

In fact, we can consider an even more general model defined below.

{\bf Model $M_k$}: For integer $k\geq 1$, let $D_k=\left\{0, \frac{1}{k}, \frac{2}{k},\ldots,\frac{k-2}{k}, \frac{k-1}{k}, 1\right\}$ be the set of all possible values of $X_{ij}$. Each of these values can be obtained with positive probability.
For $i\neq j$, $X_{ij}+X_{ji}=1$, $X_{ij}\in D_k$. This together with the assumption that all players are
equally strong (i.e., $X_{ij}$ are equally distributed), implies
\begin{equation*}
\label{eq:sym}
0<\mathbb{P}\left(X_{ij}=a\right)=\mathbb{P}\left(X_{ij}=1-a\right)\,\,\,\, \text{for any}\,\,\,\, a\in D_k.
\end{equation*}

A possible extension of the model $M_k$, which has a finite number of payoffs, is to allow a countable number of payoffs taking values in the interval $[0,1]$. Such a model is formally described as follows.

{\bf Model $M_{[0,\,1]}$}: Let $D$ in $[0, 1]$ be the discrete set of all possible values of $X_{ij}$.
Each of these values can be obtained with positive probability.
For $i\neq j$, $X_{ij}+X_{ji}=1$, $X_{ij}\in D$. This together with the assumption that all players are
equally strong (i.e., $X_{ij}$ are equally distributed), implies 
\begin{equation}
\label{eq:symD}
0<\mathbb{P}\left(X_{ij}=a\right)=\mathbb{P}\left(X_{ij}=1-a\right)\,\,\,\, \text{for any}\,\,\,\, a\in D.
\end{equation}
We also assume that $|D|>1$.

In this paper, we show that in the round-robin tournament model with equally strong players, where scores take values on a countable set $D \subset [0, 1]$ (model $M_{[0,\,1]}$), the probability that there is a unique vertex with the maximum score tends to $1$ as $n \to \infty$. This result is more general than the corresponding result for the model $M_k$, $k \ge 1$, and includes it as a special case. In the following section, we present some preliminary results. Afterwards, we present a proof of the main result. Finally, the discussion section presents some related issues.

\section{Preliminary Results }

Let
\begin{equation*}
\label{eq:mom}
\mu={\mathbb E}(X_{ij}),\,\,\, \sigma=(Var(X_{ij}))^{1/2}.
\end{equation*}

Next, let
\begin{align}
\label{eq:t}
&
t_n=(n-1)\mu+x_n(n-1)^{1/2}\sigma,
\end{align}
where
\begin{align*}
%\label{eq:tn}
&
x_{n}=[2\log(n-1)-(1 + \epsilon)\log(\log(n-1))]^{1/2},\,\,\,\epsilon>0.
\end{align*}

Assuming \eqref{eq:symD}, we obtain $\mu = \tfrac{1}{2}$ and $\sigma \le \tfrac{1}{2}$.

\begin{pro}
\label{eq:Pro1}
Assuming Model $M_{[0,\,1]}$, we have
\begin{align*}
%\label{eq:max}
\lim_{n\rightarrow \infty}\mathbb{P}\left(s^{\star}_n>t_n \right)=1.
\end{align*}
\end{pro}
\begin{proof}
Since ${\displaystyle x_{n}=o\left((n-1)^{1/6}\right)}$,
${\displaystyle \frac{x_{n}^{3}}{(n-1)^{1/2}}\rightarrow 0}$ and $\sigma(n-1)^{1/2}\rightarrow \infty$, as $n\rightarrow \infty$, it follows
from \cite[p. 553, Theorem 3]{F1971} that

\begin{align}
\label{Fe}
&
\mathbb{P}(s_1(n)>t_n)\sim 1-\Phi(x_n),
\end{align}
where $\Phi()$ is the standard normal distribution function. 
Also, since $x_n\rightarrow \infty$ as $n\rightarrow \infty$ (see for example, \cite[p. 175, Lemma 2]{F1968}),
\begin{equation}
\label{eq:MR}
1-\Phi(x_n)\sim \frac{1}{x_n}\varphi(x_n),
\end{equation}
where $\varphi()$ is the PDF of a standard normal random variable.
Combining \eqref{Fe} and \eqref{eq:MR}, we obtain
\begin{align}
\label{eq:tail}
\mathbb{P}(s_1(n)>t_n)\sim \frac{(\log(n-1))^{\epsilon/2}}{\sqrt{4\pi}(n-1)}.
\end{align}
It was shown in \cite{MR2022} that the scores 
$s_1(n), s_2(n), \ldots, s_n(n)$ 
are negatively associated (see \cite{JP1983} for the definition). 
This implies that 
\begin{equation}
\label{eq:MMR}
\mathbb{P}\left(s_1(n)\leq x, s_2(n)\leq x, \ldots,s_n(n)\leq x\right)\leq
\mathbb{P}\left(s_1(n)\leq x\right)\mathbb{P}\left(s_2(n)\leq x\right)\cdots \mathbb{P}\left(s_n(n)\leq x\right).
\end{equation}
Taking into account \eqref{eq:MMR}, the inequality $1 - c \le e^{-c}$, and the fact that the scores are equally distributed, we obtain
\begin{align}
\label{eq:B}
&
\mathbb{P}\left(s^{\star}_n>t_n \right)=1-\mathbb{P}\left(s^{\star}_n\leq t_n \right)\geq 1-\left(1-\mathbb{P}\left(s_1(n)>t_n \right)\right)^n \geq 1-e^{-n\mathbb{P}\left(s_1(n)>t_n \right)}.
\end{align}

Finally, combining \eqref{eq:tail} and \eqref{eq:B}, we obtain
\begin{align*}
%\label{eq:max}
\lim_{n\rightarrow \infty}\mathbb{P}\left(s^{\star}_n>t_n \right)=1.
\end{align*}
\end{proof}

Recall the definition of $t_n$ in \eqref{eq:t} and define
\begin{equation}
\label{eq:set}
A_n = \left\{ x:\,\,t_n-1<x\leq n-2\,\,\,\text{and}\,\, x\,\, \text{ is an attainable value of } s_u(n-1) \right\}.
\end{equation}

Recalling the definition of $W_n(t)$ in \eqref{eq:W}, we have the following proposition.

\begin{pro}
\label{eq:A}
Assuming Model $M_{[0,\,1]}$ and $n\geq 3$, we have
\begin{align*}
&
{\mathbb E}\left(W_n(t_n)\right)
\leq
\frac{n(n-1)}{2}\sum_{h\in A_n} \mathbb{P}\left(s_{u}(n-1)=h\right)\mathbb{P}\left(s_{v}(n-1)=h\right),
\end{align*}
where the set $A_n$ is defined in~\eqref{eq:set}.
\end{pro}

\begin{proof}
By the notation $\{x :\,\,\,\, \}$ we mean the attainable values of the corresponding scores only.
From the fact that the scores are equally distributed and by the law of total probability, we obtain, recalling \eqref{eq:W},

\begin{align}
\label{eq:RHS}
&
{\mathbb E}\left(W_n(t_n)\right)=\frac{n(n-1)}{2}\mathbb{P}\left(t_n< s_{u}(n)=s_{v}(n)\right)\nonumber\\
&
=\frac{n(n-1)}{2}
\sum_{{{h\in \left\{x:\,\, t_{n}<x\leq n-1\right\}} } }
\mathbb{P}\left(s_{u}(n)=h, s_{v}(n)=h\right)\nonumber\\
&
=\frac{n(n-1)}{2}\sum_{a\in D}\sum_{{{h\in \left\{x:\,\, t_{n}<a+x\leq a+n-2\right\}} }}\mathbb{P}\left(s_{u}(n)=h, s_{v}(n)=h\,|\, X_{uv}=a\right) \mathbb{P}\left(X_{uv}=a\right)\nonumber\\
&
=\frac{n(n-1)}{2}\sum_{a\in D}\sum_{{{h\in \left\{x:\,\, t_{n}<a+x\leq a+n-2\right\}} }} \mathbb{P}\left(s_{u}(n-1)=h-a\right)\mathbb{P}\left(s_{v}(n-1)=h-(1-a)\right)\mathbb{P}\left(X_{uv}=a\right)\nonumber\\
&
=\frac{n(n-1)}{2}\sum_{a\in D}\mathbb{P}\left(X_{uv}=a\right)\sum_{{{h\in \left\{x:\,\, t_{n}<a+x\leq a+n-2\right\}} }}\mathbb{P}\left(s_{u}(n-1)=h-a\right)\mathbb{P}\left(s_{v}(n-1)=h-(1-a)\right)\nonumber\\
&
\leq \frac{n(n-1)}{2}\sum_{a\in D }\mathbb{P}\left(X_{uv}=a\right)\sum_{{{h\in \left\{x:\,\, t_{n}<a+x\leq a+n-2\right\}} }} \mathbb{P}\left(s_{u}(n-1)=h-a\right)\mathbb{P}\left(s_{v}(n-1)=h-a\right)\nonumber\\
&
\leq
\frac{n(n-1)}{2}\sum_{a\in D }\mathbb{P}\left(X_{uv}=a\right)\sum_{{{h\in \left\{x:\,\, t_{n}-1<x\leq n-2\right\}} }} \mathbb{P}\left(s_{u}(n-1)=h\right)\mathbb{P}\left(s_{v}(n-1)=h\right)\nonumber\\
&
=
\frac{n(n-1)}{2}\sum_{h\in A_n} \mathbb{P}\left(s_{u}(n-1)=h\right)\mathbb{P}\left(s_{v}(n-1)=h\right),
\end{align}
where the first inequality follows from combining assumption \eqref{eq:symD} with the rearrangement inequality, and the second inequality follows from adding nonnegative summands.
\end{proof}

\begin{pro}
\label{eq:Pro3}
Assuming Model $M_{[0,\,1]}$, we have
\begin{equation}
{\displaystyle \lim_{n\rightarrow \infty}\mathbb{P}\left(W_n(t_n)=0\right)=1.}
\end{equation}
\end{pro}

\begin{proof} 

First we need the following Claim.
\begin{cl}
\label{eq:cl1}
For any $\varepsilon > 0$ and sufficiently large $n$,
\begin{align*}
\sup_{x \in A_n} \mathbb{P}(s_1(n-1) = x) 
\le O\Bigg(\frac{(\log n)^{(1+\varepsilon)/2}}{n^{3/2}}\Bigg).
\end{align*}
\end{cl}

\begin{proof}{[Claim \ref{eq:cl1}]}
For any $\varepsilon > 0$, it follows from \eqref{eq:t} that  
\begin{equation}
\label{eq:L}
t_n - 1 \ge (n-2)\mu + x_{n} (n-2)^{1/2}:=l_n.
\end{equation}

For simplicity, let us denote $m = n - 2$ and $S_m = s_1(n-1)$.

Define, for $t\in\mathbb{R}$,
$$
\Pr_{t}(S_m=x)
:=\frac{\mathbb{E}\!\left(e^{tS_m}\mathbf{1}_{\{S_m=x\}}\right)}
{\mathbb{E}\!\left(e^{tS_m}\right)}.
$$

Then, for any $t\in\mathbb{R}$,

\[
\begin{aligned}
\mathbb{P}(S_m=x)
&=e^{-tx}\,\mathbb{E}\!\left(e^{tS_m}\mathbf{1}_{\{S_m=x\}}\right)=
e^{-tx}\,\mathbb{E}\!\left(e^{tS_m}\right)\,
   \frac{\mathbb{E}\!\left(e^{tS_m}\mathbf{1}_{\{S_m=x\}}\right)}
        {\mathbb{E}\!\left(e^{tS_m}\right)}\\[4pt]
&=e^{-tx}\,\mathbb{E}\!\left(e^{tS_m}\right)\,
   \Pr_t(S_m=x).
\end{aligned}
\]

For $t > 0$ and $x > l_n$, we have

\[
\mathbb{P}(S_m=x)
= e^{-tx}\,\mathbb{E}\!\left(e^{tS_m}\mathbf{1}_{\{S_m=x\}}\right)
\le e^{-t l_n}\,\mathbb{E}\!\left(e^{tS_m}\right)\Pr_t(S_m=x),
\]
where the inequality follows from the fact that $e^{-tx}$ is non-increasing in $x$ for $t>0$.

Therefore, for $t>0$,
\begin{equation}
\label{eq:sup}
\sup_{x>t_n-1}\mathbb{P}(S_m=x)
\le
\sup_{x>l_n}\mathbb{P}(S_m=x)
\le
e^{-t l_n}\,\mathbb{E}\!\left(e^{tS_m}\right)
\sup_{x>l_n}\Pr_t(S_m=x).
\end{equation}

\item
We have
$
\mathbb{E}\!\left(e^{tS_m}\right) = \big(\mathbb{E}\!\left(e^{tX_1}\right)\big)^m = \big(M_{X_1}(t)\big)^m.
$
Denote
\begin{equation*}
K(t) := \log M_{X_1}(t),
\end{equation*}
so that
\begin{equation*}
\log \mathbb{E}\!\left(e^{t S_m}\right) = m K(t),
\end{equation*}
and $K^{(j)}(t)$ denotes the $j$-th derivative of $K(t)$ with respect to $t$.

Using the Taylor expansion about $t=0$ and noting that for the symmetric model \eqref{eq:symD} with support on $[0,1]$, $K^{(3)}(0)=0$, we have
\begin{equation*}
K(t) = \mu t + \frac{\sigma^2}{2} t^2 + R(t),
\end{equation*}
where $\mu = K^{(1)}(0) = \mathbb{E}\left(X_1\right)$, $\sigma^2 = K^{(2)}(0) = \mathrm{Var}(X_1)$, and the remainder term is
\begin{equation*}
R(t) = \frac{K^{(4)}(\eta)}{24} t^4 \le c\, t^4, 
\end{equation*}
for some $0 < \eta < t$ and some positive constant $c$.

Therefore, we have
\begin{align*}
\log \mathbb{E}(e^{tS_m}) &= m K(t) = m \mu t + m \frac{\sigma^2}{2} t^2 + m R(t),
\end{align*}

and
\begin{equation*}
e^{-t l_n} \mathbb{E}\!\left(e^{t S_m}\right) 
= e^{\frac{m \sigma^2 t^2}{2} - t \sigma \sqrt{m} x_{n} + m R(t)}
= e^{\frac{m \sigma^2 t^2}{2} - t \sigma \sqrt{m} x_{n} + O(m t^4)},
\end{equation*}
where the last equality uses $R(t) = O(t^4)$.

Choose 
\begin{equation*}
t^{\star} = \frac{x_{n}}{\sigma \sqrt{m}}.
\end{equation*}
Then 
\begin{equation*}
O(m \left(t^{\star}\right)^4) = O\Big(\frac{(\log m)^2}{m}\Big),
\end{equation*}
and we obtain
\begin{align*}
\frac{m \sigma^2 (t^\star)^2}{2} - t^\star \sigma \sqrt{m} x_n + O(m t^4)
&= \frac{m \sigma^2 \frac{x_n^2}{\sigma^2 m}}{2} - \frac{x_n}{\sigma \sqrt{m}} \sigma \sqrt{m} x_n + O(m \left(t^{\star}\right)^4) \nonumber\\
&= -\frac{x_n^2}{2} + o(1).
\end{align*}

Therefore, for large $n$ and $t>0$, 
\begin{align}
\label{eq:first}
e^{-t l_n} \mathbb{E}\!\left(e^{t S_m}\right) 
\le e^{\Big(-\frac{x_n^2}{2} + o(1)\Big)}
\le \frac{(\log n)^{(1+\varepsilon)/2}}{n} \,(1 + o(1)).
\end{align}

For the evaluation of the last term $\sup_{x>l_n}\Pr_t(S_m=x)$ in \eqref{eq:sup}, we will need the following theorem due to \cite{R1961}, who proved Kolmogorov's inequality on the rate of decrease of Lévy's concentration function for the sum of independent discrete random variables (see also \cite{P1975}, pages 56 and 319):

\begin{thm}[\cite{R1961}]
Let $X_1,\ldots,X_n$ be independent discrete random variables, and let $p_j = \sup_{x} \mathbb{P}\left(X_j = x\right).$
Then
$
\sup_{x} \mathbb{P}(\sum_{j=1}^{n}X_j = x) \le \frac{A}{\sqrt{\sum_{j=1}^{n} (1 - p_j)}},
$
for some constant $A$.
\end{thm}

Denote
\begin{equation*}
p_{\max} = \sup_{a} \Pr_t(X_1 = a)<1.
\end{equation*}

From Rogozin's Theorem, it immediately follows that
\begin{align}
\label{eq:second}
\sup_{x > l_n} \Pr_t(S_n = x) 
&\le \sup_{x} \Pr_t(S_n = x) 
\le \frac{A}{\sqrt{1-p_{\max}}\sqrt{m}}.
\end{align}

Combining \eqref{eq:first} and \eqref{eq:second}, we obtain that for any $\varepsilon > 0$ and sufficiently large $n$,
\begin{align*}
\sup_{x \in A_n} \mathbb{P}(s_1(n-1) = x) 
\le O\Bigg(\frac{(\log n)^{(1+\varepsilon)/2}}{n^{3/2}}\Bigg).
\end{align*}

\end{proof}

Now, by Proposition~\ref{eq:A}, and Claim \ref{eq:cl1} , as $n \to \infty,$
\begin{align}
\label{eq:last}
&
{\mathbb E}\left(W_n(t_n)\right)
\leq
\frac{n(n-1)}{2}\sum_{h\in A_n} \mathbb{P}\left(s_{u}(n-1)=h\right)\mathbb{P}\left(s_{v}(n-1)=h\right)\nonumber\\
&
\leq\frac{n(n-1)}{2}\sup_{x \in A_n} \mathbb{P}(s_1(n-1) = x) \sum_{h\in A_n} \mathbb{P}\left(s_{v}(n-1)=h\right)\nonumber\\
&
=
\frac{n(n-1)}{2}\sup_{x \in A_n} \mathbb{P}(s_1(n-1) = x) \mathbb{P}\left(s_{v}(n-1)> t_{n}-1\right)\nonumber\\
&
\leq
\frac{n(n-1)}{2}
\frac{1}{\sigma\sqrt{n-2}}O\Bigg(\frac{(\log n)^{(1+\varepsilon)/2}}{n^{3/2}}\Bigg)
\mathbb{P}\left(s_{v}(n-1)> t_{n}-1\right)\nonumber\\
&\sim 
\frac{n(n-1)}{2}
O\Bigg(\frac{(\log n)^{(1+\varepsilon)/2}}{n^{3/2}}\Bigg)
\frac{(\log(n-1))^{\epsilon/2}}{\sqrt{4\pi}(n-1)},
\end{align}
where the last line is obtained in a manner similar to~\eqref{eq:tail}.

Since $W_n(t_n) \ge 0$, and by \eqref{eq:last}, it follows that
\begin{equation*}
\label{eq:EW}
{\displaystyle \lim_{n\rightarrow \infty}{\mathbb E}\left(W_n(t_n)\right)=0.}
\end{equation*}
\end{proof}

\section{Main Result}

\begin{result}
\label{eq:Main}
\begin{equation}
{\displaystyle \lim_{n\rightarrow \infty}r_n(M_{[0,\,1]})=1.}
\end{equation}
\end{result}
\begin{proof}
Appealing to Proposition~\ref{eq:Pro3} and using the inequality 
$W_n = W_n I(W_n > 0) \ge I(W_n > 0)\ge 0$, we obtain \eqref{eq:Step2} under Model~$M_{[0,\,1]}$, i.e.,
\begin{equation}
\label{eq:Step22}
\lim_{n \to \infty} \mathbb{P}\bigl(W_n(t_n) = 0\bigr) = 1.
\end{equation}

Combining Propositions~\ref{eq:Pro1} and \eqref{eq:Step22}, we obtain Result~\ref{eq:Main}.
\end{proof}

\section{Discussion}

The main result for the model $M_{[0,\,1]}$ also holds for the model $M_k$, and in particular Claim \ref{eq:cl1} with a precise rate can be obtained via a local limit theorem, as is classical in number theory \citep{AF1988} using the circle method. Our results also hold for general score distributions on $[0,1]$. Indeed, setting 
$D = \{x \in [0,1] : \mathbb{P}(X = x) > 0\}$, if $\mathbb{P}(X \in D) < 1$, 
it follows by straightforward considerations that the maximum is asymptotically unique.

\section*{Acknowledgements}
YM would like to thank Noga Alon for valuable discussions and suggestions, and Doron Zeilberger for the opportunity to present this work in his seminar and for his insightful questions.
We thank the referee for a careful reading of the manuscript and for helpful comments, questions, and suggestions that significantly improved the clarity and quality of the paper.
Research of GA was supported in part by BSF grant 2024051. 
Research of YM was supported in part by BSF grant 2020063.

\end{document}